\documentclass[a4paper,12pt]{article}
\usepackage{amsmath}
\usepackage{amsfonts}
\usepackage{amssymb}
\usepackage{amsthm}
\usepackage{bbm}
\usepackage[latin1]{inputenc}
\usepackage{graphics}
\usepackage{epsfig}
\usepackage{color}
\usepackage{vmargin}
\usepackage{color}

\newtheorem{thm}{Theorem}
\newtheorem{prop}[thm]{Proposition}

\newtheorem{cor}[thm]{Corollary}

\renewcommand{\P}{\mathbbm{P}}
\def\esp{\mathbbm{E}}

\DeclareMathOperator{\Var}{Var}

\DeclareMathOperator{\Tr}{Tr}

\title{Eigenvalue variance bounds for Wigner and covariance random matrices}
\author{S. Dallaporta}
\date {University of Toulouse, France}

\begin{document}
\maketitle

\bigskip 
{{\bf Abstract.} 
{\it This work is concerned with finite range bounds on the variance of individual eigenvalues
of Wigner random matrices, in the bulk and at the edge of the spectrum, as well as for some intermediate
eigenvalues. Relying on the GUE example, which needs to be investigated first, the main bounds are
extended to families of Hermitian
Wigner matrices by means of the Tao and Vu Four Moment Theorem and recent localization
results by Erdös, Yau and Yin. The case of real Wigner matrices is obtained from interlacing formulas.
As an application, bounds on the expected $2$-Wasserstein distance between the
empirical spectral measure and the semicircle law are derived. Similar results are available for random covariance matrices.}

\bigskip \bigskip \bigskip


Two different models of random Hermitian matrices were introduced by Wishart in the twenties and by Wigner in
the fifties. Wishart was interested in modeling tables of random data in multivariate analysis and worked on
random covariance matrices. In this paper, the results for covariance matrices are very close to those for
Wigner matrices. Therefore, it deals mainly with Wigner matrices. Definitions and results regarding covariance
matrices are available in the last section.

Random Wigner matrices were first introduced by Wigner to study eigenvalues of infinite-dimensional operators
in statistical physics (see \cite{Me_1991_book}) and then propagated to various fields of mathematics involved
in the study of spectra of random matrices. Under suitable symmetry assumptions, the asymptotic properties of
the eigenvalues of a random matrix were soon conjectured to be universal, in the sense they do not depend on
the individual distribution of the matrix entries. This opened the way to numerous developments on the
asymptotics of various statistics of the eigenvalues of random matrices, such as for example the global
behavior of the spectrum, the spacings between the eigenvalues in the bulk of the spectrum or the behavior of
the extreme eigenvalues. 
Two main models have been considered, invariant matrices and Wigner matrices.
In the invariant matrix models, the matrix law is unitary
invariant and the eigenvalue joint distribution can be written explicitly in terms of a
given potential. In the
Wigner models, the matrix entries are independent (up to symmetry conditions). The case where the entries are
Gaussian is the only model belonging to both types. In the latter case,
the joint distribution of the eigenvalues is thus
explicitly known and the previous statistics have been completely studied.
One main focus of random
matrix theory in the past decades was to prove that these asymptotic behaviors were the same for non-Gaussian
matrices (see for instance \cite{AnGuZe_2010_book}, \cite{BaSi_2010_book} and \cite{PaSh_2011_book}).

However, in several fields such as
computer science or statistics for example, asymptotic statements are often
not enough, and more quantitative finite range results are required.
Several recent developments have thus been concerned with non-asymptotic random
matrix theory towards quantitative bounds
(for instance on the probability for a certain event to occur) which are valid for all $N$,
where $N$ is the size of the given matrix. See for example \cite{Ve_2012_non_asymptotic} for an introduction to some problems considered in non-asymptotic random matrix theory.
In this paper, we investigate in this respect variance bounds on the eigenvalues of
families of Wigner random matrices. 

Wigner matrices are Hermitian or real symmetric
matrices $M_N$ such that, if
$M_N$ is complex, for $i<j$, the real and imaginary parts of $(M_N)_{ij}$ are independent and identically
distributed (iid) with mean $0$ and variance
$\frac{1}{2}$, $(M_N)_{ii}$ are iid, with mean $0$ and variance $1$. In the real case, $(M_n)_{ij}$ are iid,
with mean $0$ and variance $1$ and $(M_N)_{ii}$ are iid, with mean $0$ and variance $2$. In both cases, set
$W_N=\frac{1}{\sqrt{N}}M_N$.
An important example of Wigner matrices is the case where the entries are Gaussian. If $M_N$ is complex, then
it belongs to the so-called Gaussian Unitary Ensemble (GUE). If it is real, it belongs to the Gaussian
Orthogonal Ensemble (GOE).
The matrix $W_N$ has $N$ real eigenvalues $\lambda_1 \leqslant \cdots \leqslant \lambda_N$.
In the Gaussian case, the
joint law of the eigenvalues is known, allowing for complete descriptions of their limiting behavior both in
the global and local regimes (see for example \cite{AnGuZe_2010_book}, \cite{BaSi_2010_book} and \cite{PaSh_2011_book}).

Among universality results, at the global level, the classical Wigner's Theorem states that the empirical
distribution $L_N=\frac{1}{N}\sum_{j=1}^N\delta_{\lambda_j}$ on the eigenvalues of $W_N$ converges
weakly almost surely to the semicircle law $ d \rho_{sc}(x) =
\frac{1}{2\pi}\sqrt{4-x^2}\mathbbm{1}_{[-2,2]}(x) dx$ (see for example \cite{BaSi_2010_book} for a proof in
the more general setting). This gives the global asymptotic behavior of the spectrum.
However, this theorem is not enough to deduce some information on individual eigenvalues. Define, for
all $1\leqslant j \leqslant N$, the theoretical
location of the $j^{\textrm{th}}$ eigenvalue $\gamma_j$ by
$\int_{-2}^{\gamma_j}d\rho_{sc}(x)=\frac{j}{N}$. Bai and Yin proved in \cite{BaYi_1988_largest} with
assumptions on higher moments that almost surely the smallest and the largest eigenvalues converge to their
theoretical locations, which means that $\lambda_N \to 2$ and $\lambda_1 \to -2$ almost surely (see also
\cite{BaSi_2010_book} for a proof of this theorem). From this almost sure convergence of the extreme
eigenvalues and Wigner's Theorem, it is possible to deduce information on individual eigenvalues in the bulk
of the spectrum. Indeed, according to the Glivenko-Cantelli Theorem (see for example
\cite{Du_2002_analysis_proba}), the normalized eigenvalue function $\frac{1}{N}\mathcal{N}_x$, where
$\mathcal{N}_x$ is the number of eigenvalues which are in $(-\infty,x]$, converges uniformly on $\mathbb{R}$
almost surely to the distribution function of the semicircle law $G$ (with no more assumptions on the matrix
entries). Then, using that $\frac{1}{N}\mathcal{N}_{\lambda_j}=\frac{j}{N}=G(\gamma_j)$
together with crude bounds on the semicircle density function and the fact that $\lambda_j$ is almost surely
between $-2-\varepsilon$ and $2+\varepsilon$ shows that almost surely $\lambda_j-\gamma_j \to 0$ uniformly for
$\eta N \leqslant j \leqslant (1-\eta)N$ (for any fixed $\eta>0$).

At the fluctuation level, eigenvalues inside the bulk and at the edge of the spectrum do not have the same behavior.
Tracy and Widom showed in \cite{TrWi_1994_airy} that the largest eigenvalue fluctuates around $2$ according to
the so-called Tracy-Widom law $F_2$. Namely,
\[ N^{2/3}(\lambda_N-2) \to F_2 \]
in distribution as $N$ goes to infinity. 
They proved this result for Gaussian matrices of the GUE,
later extended to families of non-Gaussian Wigner matrices by
Soshnikov (see \cite{So_1999_universality}). Recent results by Tao and Vu (see \cite{TaVu_2010_wigner_edge}) and
by Erd\"{o}s, Yau and Yin (see \cite{ErYaYi_2010_rigidity}) provide alternate proofs of this fact for larger
families of Wigner matrices. According to this asymptotic property, the variance of the largest eigenvalue
$\lambda_N$ is thus of the order
of $N^{-4/3}$. In the bulk, Gustavsson proved in \cite{Gu_2005_fluctuations}
again for the GUE, that, for any fixed $\eta>0$ and
all $\eta N \leqslant j \leqslant (1-\eta)N$,
\begin{equation}\label{gustavsson_bulk}
 \frac{\lambda_j-\gamma_j}{\sqrt{\frac{2\log N}{(4-\gamma_j^2)N^2}}} \to \mathcal{N}(0,1)
\end{equation}
in
distribution as $N$ goes to infinity. This result was extended by Tao and Vu in
\cite{TaVu_2011_wigner} to large families of non-Gaussian Wigner
matrices. The variance of an eigenvalue $\lambda_j$ in the bulk is thus of the order of
$\frac{\log N}{N^2}$.
Right-side intermediate eigenvalues consist in the $\lambda_j$'s with $\frac{N}{2} \leqslant j \leqslant N$
such
that~$\frac{j}{N} \to 1$
but~$N-j \to \infty$ as $N$ goes to infinity (the left-side can be deduced by
symmetry). Gustavsson proved a Central Limit Theorem for these eigenvalues (see
\cite{Gu_2005_fluctuations}) from which their variance is guessed to be of the order of
$\frac{\log(N-j)}{N^{4/3}(N-j)^{2/3}}$. This result was again
extended to large classes of Wigner matrices by Tao
and Vu in \cite{TaVu_2010_wigner_edge}.

The previous results are however asymptotic. As announced, the
purpose of this work is to provide quantitative bounds on the variance of the eigenvalues
of the correct order, in the bulk and at the edge of the spectrum, as well as for some intermediate eigenvalues. 
In the statements below, $M_N$ is a complex (respectively real) Wigner
matrix satisfying condition $(C0)$. This condition, which will be detailed in Section \ref{tools_Wigner},
provides an exponential decay of the matrix entries. Assume furthermore that the entries
of $M_N$ have the same first
four moments as the entries of a GUE matrix (respectively GOE). Set $W_N=\frac{1}{\sqrt{N}}M_N$.

\begin{thm}[in the bulk]\label{thm_bulk}
For all $0 <\eta \leqslant \frac{1}{2}$, there exists a constant $C(\eta)>0$ such that, for all $N \geqslant 2$ and for all $\eta N
\leqslant j \leqslant (1-\eta)N$,
\begin{equation}
\Var(\lambda_j) \leqslant C(\eta)\frac{\log N}{N^2}.
\end{equation}
\end{thm}

\begin{thm}[at the edge]\label{thm_edge}
There exists a universal constant $C>0$ such that, for all $N \geqslant 1$,
\begin{equation}
\Var(\lambda_1) \leqslant CN^{-4/3} \quad \textrm{and} \quad \Var(\lambda_N) \leqslant CN^{-4/3}.
\end{equation}

\end{thm}

\begin{thm}[between the bulk and the edge]\label{thm_intermediate}
For all $0<\eta \leqslant \frac{1}{2}$ and for all $K>20\sqrt{2}$, there exists a constant $C(\eta,K)>0$ such that,
for all $N \geqslant 2$, for all $K\log N \leqslant j \leqslant \eta N$ and $(1-\eta)N \leqslant j \leqslant
N-K\log N$,
\begin{equation}
\Var(\lambda_j) \leqslant C(\eta,K)\frac{\log\big(\min(j,N-j)\big)}{N^{4/3}\big(\min(j,N-j)\big)^{2/3}}.
\end{equation}
\end{thm}

It should be mentioned that
Theorem \ref{thm_bulk} does not seem to be known even for Gaussian matrices.
The first step is thus to prove it for the GUE.
This will be achieved via the analysis of the eigenvalue counting function, which
due to the particular determinantal structure in this case, has the same distribution as a sum of independent
Bernoulli variables.
Sharp standard deviation inequalities are thus available in this case. These may then
be transferred to the eigenvalues in the bulk together with Gustavsson's bounds
on the variance of the eigenvalue counting function. As a result, we actually establish that
\[ E\big [ | \lambda_j - \gamma_j|^2 \big] \leqslant C(\eta)\frac{\log N}{N^2} \]
leading thus to Theorem \ref{thm_bulk} in this case. Similarly, Theorem \ref{thm_intermediate} does not seem to be known for Gaussian matrices. The proof follows exactly the same scheme and we establish that
\[E\big [ | \lambda_j - \gamma_j|^2 \big] \leqslant C(\eta,K)\frac{\log\big(\min(j,N-j)\big)}{N^{4/3}\big(\min(j,N-j)\big)^{2/3}}.\]
On the other hand,
Theorem \ref{thm_edge} for the GUE and GOE has been known for some time (see \cite{LeRi_2010_deviations}). 
On the basis of these results for the GUE (and GOE), Theorems \ref{thm_bulk}, \ref{thm_edge} and
\ref{thm_intermediate} are then extended to families of Wigner matrices by a suitable
combination of Tao and Vu's Four Moment Theorem (see \cite{TaVu_2011_wigner}, \cite{TaVu_2010_wigner_edge}) and
Erd\"os, Yau and Yin's Localization Theorem (see \cite{ErYaYi_2010_rigidity}). The basic idea is that
while the localization properties almost yield the correct order, the
Four Moment Theorem may be used to reach the optimal bounds by comparison
with the Gaussian models. Theorems \ref{thm_bulk}, \ref{thm_edge} and \ref{thm_intermediate} are established
first in the complex case. The real case is deduced by means of
interlacing formulas. Furthermore, analogous results are established for covariance matrices, for which
non-asymptotic quantitative results are needed, as they are useful in several fields,
such as compressed sensing (see \cite{Ve_2012_non_asymptotic}), wireless communication and quantitative
finance (see \cite{BaSi_2010_book}).

The method developed here do not seem powerful enough to strengthen
the variance bounds into exponential tail inequalities. In the recent contribution
\cite{TaVu_2012_deviation}, Tao and Vu proved such
exponential tail inequalities by a further refinement of the replacement method
leading to the Four Moment Theorem. While much more powerful than variance
bounds, they do not seem to yield at this moment the correct order of the variance bounds
of Theorems \ref{thm_bulk}, \ref{thm_edge} and \ref{thm_intermediate}.

As a corollary of the latter results, a bound
on the rate of convergence of the empirical spectral measure $L_N$ can be
achieved. This bound is expressed in terms of the so-called
$2$-Wasserstein distance $W_2$ between $L_N$ and the semicircle
law $\rho_{sc}$. For $p \in [1,\infty)$, the $p$-Wasserstein distance $W_p(\mu,\nu)$ between two probability measures $\mu$ and $\nu$ on $\mathbb{R}$ is defined by
\[W_p(\mu,\nu)=\inf \bigg (\int_{\mathbb{R}^2}|x-y|^p\,d\pi(x,y)\bigg)^{1/p}\]
where the infimum is taken over all probability measure $\pi$ on $\mathbb{R}^2$
such that its first marginal is $\mu$ and its second marginal is $\nu$.
Note for further purposes that the rate of convergence of this empirical distribution
has been investigated in various forms. For example, the
Kolmogorov distance between $L_N$ and $\rho_{sc}$
has been considered in this respect
in several papers (see for example \cite{GoTi_2003_convergence}, 
\cite{BoGoTi_2010_convergence_semicircle} and \cite{GoTi_2011_convergence}). It is given by
\[d_K(L_N,\rho_{sc})= \sup_{x \in \mathbb{R}}
\big|\tfrac{1}{N}\mathcal{N}_x-G(x)\big|,\] 
where $\mathcal{N}_x$ is the eigenvalue counting function and $G$ is the distribution function
of the semicircle law. More and
more precise bounds were established. For example, under the hypothesis of an exponential decay of the matrix
entries, Götze and Tikhomirov recently showed that, with high probability,
\[d_K(L_N,\rho_{sc}) \leqslant \frac{(\log N)^c}{N}\]
for some universal constant $c>0$ (see \cite{GoTi_2011_convergence}). 
The rate of convergence in terms of $W_1$, also called the Kantorovich-Rubinstein distance, was studied by Guionnet and Zeitouni in \cite{GuZe_2000_concentration_spectral_measure} and recently
by Meckes and Meckes in \cite{MeMe_2011_convergence_spectral_measure}. In
\cite{MeMe_2011_convergence_spectral_measure}, the authors proved that $\esp[W_1(L_N,\esp[L_N])] \leqslant
CN^{-2/3}$. The following is concerned with 
the distance between $L_N$ and $\rho_{sc}$ and strengthens the preceding conclusion on
$W_1$.
\begin{cor}\label{cor_wasserstein_2} There exists a 
numerical constant $C>0$ such that, for all $N \geqslant 2$,
\begin{equation}
\esp\big[W_2^2(L_N,\rho_{sc})\big]\leqslant C\frac{\log N}{N^2}.
\end{equation}
\end{cor}
The proof of this corollary relies on the fact that $\esp\big[W_2^2(L_N,\rho_{sc})\big]$
is bounded above, up to a constant, by the sum of the expectations
$\esp\big[(\lambda_j-\gamma_j)^2\big]$. The previously established bounds
then easily yield the result.

Turning to the organization of the paper, Section~\ref {GUE} describes Theorems \ref{thm_bulk}, \ref{thm_edge}
and \ref{thm_intermediate} in the GUE case. Section~\ref{Wigner} emphasizes to start with the Four Moment
Theorem of Tao and Vu (see \cite{TaVu_2011_wigner} and \cite{TaVu_2010_wigner_edge}) and the Localization
Theorem of Erdös, Yau and Yin (see \cite{ErYaYi_2010_rigidity}). On the basis of these results and the GUE
case, the main Theorems \ref{thm_bulk}, \ref{thm_edge} and \ref{thm_intermediate} are then established for
families of Wigner matrices. Section~\ref{real} is devoted to the corresponding statements for real matrices,
while Section~\ref{covariance} describes the analogous results for covariance matrices.
Section~\ref{wasserstein} deals with Corollary \ref {cor_wasserstein_2} and the rate of convergence of $L_N$
towards $\rho_{sc}$ in terms of $2$-Wasserstein distance.

Throughout this paper, $C(\cdot)$ and $c(\cdot)$ will denote positive constants which depend on the specified
quantities and may differ from one line to another.

\section{Deviation inequalities and variance bounds in the GUE case} \label {GUE}

The results and proofs developped in this section for the GUE model
heavily rely on its determinantal structure which allows for a complete
description of the joint law of the eigenvalues (see \cite{AnGuZe_2010_book}). In particular the 
eigenvalue counting function is known to have the same distribution as a sum of independent Bernoulli
variables (see \cite{BeKrPeVi_2006_determinantal}, \cite{AnGuZe_2010_book}), whose mean and variance were
computed by Gustavsson (see \cite{Gu_2005_fluctuations}). Deviation inequalities for individual eigenvalues
can thus be established, leading to the announced bounds on the variance.

\subsection{Eigenvalues in the bulk of the spectrum}\label{bulk_GUE}

The aim of this section is to establish the following result.

\begin{thm}\label{thm_GUE_bulk}
Let $M_N$ be a GUE matrix. Set $W_N=\frac{1}{\sqrt{N}}M_N$.
For any $0<\eta \leqslant \frac{1}{2}$, there exists a constant $C(\eta)>0$
such that for all $N \geqslant 2$ and all $\eta N
\leqslant j \leqslant (1-\eta)N$,
\begin{equation}
\esp\big[|\lambda_j-\gamma_j|^2\big] \leqslant C(\eta)\frac{\log
N}{N^2}.
\end{equation}
In particular,
\begin{equation}
\Var(\lambda_j) \leqslant C(\eta)\frac{\log N}{N^2}.
\end{equation}
\end{thm}

As announced, the proof is based on the connection between the distribution of
eigenvalues and the eigenvalue counting function.
For every $t \in \mathbb{R}$, let $\mathcal{N}_t=\sum_{i=1}^N\mathbbm{1}_{\lambda_i \leqslant t}$
be the eigenvalue counting function. Due to the
determinantal structure of the GUE, it is known (see \cite{BeKrPeVi_2006_determinantal})
that $\mathcal{N}_t$ has the same distribution as a sum of independent Bernoulli random variables.
Bernstein's inequality for example (although other, even sharper, inequalities may be used)
may then be applied to get that for every $u \geqslant 0$,
\begin{equation}\label{bernstein}
 \P\Big(\big|\mathcal{N}_t-\esp[\mathcal{N}_t]\big|\geqslant u\Big)\leqslant 2\exp \Big
(-\frac{u^2}{2\sigma_t^2+u}\Big),
\end{equation}
where $\sigma_t^2$ is the variance of $\mathcal{N}_t$ (see for example \cite{Va_1998_asymptotic_statistics}).
Note that the upper-bound is non-increasing in $u$
while non-decreasing in the variance.
Set for simplicity $\rho_t=\rho_{sc}\big((-\infty,t]\big)$, $t \in \mathbb{R}$. It has been
shown in \cite{GoTi_2005_convergence} that for some numerical constant $C_1>0$,
\[\sup_{t \in \mathbb{R}} \big|\esp[\mathcal{N}_t]-N\rho_t\big| \leqslant C_1.\]
In particular thus, together with \eqref{bernstein}, for every $u \geqslant 0$,
\begin{equation}\label{inequality_counting_function}
 \P\Big(\big|\mathcal{N}_t-N\rho_t\big|\geqslant u+C_1\Big)\leqslant
2\exp\Big(-\frac{u^2}{2\sigma_t^2+u}\Big).
\end{equation}
As a main conclusion of the work by Gustavsson \cite{Gu_2005_fluctuations}, for every $\delta \in (0,2)$,
there exists $c_{\delta}>0$ such that
\begin{equation}\label{sup_variance}
\sup_{t \in I_{\delta}} \sigma_t^2 \leqslant c_{\delta}\log N,
\end{equation}
where $I_{\delta}=[-2+\delta,2-\delta]$.
On the basis of inequalities \eqref{inequality_counting_function} and \eqref{sup_variance},
it is then possible to derive a deviation inequality
for eigenvalues $\lambda _j$ in the bulk from their theoretical locations
$\gamma_j \in [-2,2]$, $1\leqslant j \leqslant N$, defined by
$\rho_{\gamma_j}=\frac{j}{N}$.

\begin{prop}[Deviation inequality for $\lambda_j$]\label{prop_eigenvalue_deviation_bulk}
Let $\eta \in (0,\frac{1}{2}]$ and $\eta N \leqslant j \leqslant (1-\eta) N$.
There exist $C>0$, $c>0$, $c'>0$ and $\delta \in (0,2)$ (all depending on $\eta$) such that, for all
$c \leqslant u \leqslant c'N$, 
\begin{equation}\label{deviation_bulk}
\P\Big(|\lambda_j-\gamma_j|\geqslant \frac{u}{N}\Big)\leqslant 4\exp\Big(-\frac{C^2u^2
}{2c_{\delta}\log N+Cu}\Big).
\end{equation}
\end{prop}

\begin{proof}
Let $\eta \in (0,\frac{1}{2}]$.
To start with, evaluate, for $\frac{N}{2} \leqslant j \leqslant
(1-\eta)N$ and $u \geqslant 0$,
the probability~$\P(|\lambda_j-\gamma_j|>\frac{u}{N})$. We
have
\begin{align*}
\P\Big(\lambda_j >\gamma_j+ \frac{u}{N}\Big) & = \P\Big(\sum_{i=1}^N\mathbbm{1}_{\lambda_i\leqslant
\gamma_j+\frac{u}{N}}<j\Big)\\
 & = \P\big(\mathcal{N}_{\gamma_j+\frac{u}{N}}<j\big)\\
 & = \P\big(N\rho_{\gamma_j+\frac{u}{N}}-\mathcal{N}_{\gamma_j+\frac{u}{N}}>N\rho_{\gamma_j+\frac{u}{N}}-j\big)\\
 & =
\P\big(N\rho_{\gamma_j+\frac{u}{N}}-\mathcal{N}_{\gamma_j+\frac{u}{N}}>N(\rho_{\gamma_j+\frac{u}{N}}-\rho_{
\gamma_j} )\big)
\end{align*}
where it has been used that $\rho_{\gamma_j}=\frac{j}{N}$. Then
\[\P\Big(\lambda_j >\gamma_j+ \frac{u}{N}\Big) \leqslant
\P\big(|\mathcal{N}_{\gamma_j+\frac{u}{N}}-N\rho_{\gamma_j+\frac{u}{N}}|>N(\rho_{\gamma_j+\frac{u}{N}}-\rho_{
\gamma_j } )\big).\]
But
\begin{align*}
\rho_{\gamma_j+\frac{u}{N}}-\rho_{\gamma_j} & =
\int_{\gamma_j}^{\gamma_j+\frac{u}{N}}\frac{1}{2\pi}\sqrt{4-x^2}\,dx\\
  & \geqslant \frac{1}{\sqrt{2}\pi}\int_{\gamma_j}^{\gamma_j+\frac{u}{N}} \sqrt{2-x}\,dx\\
 & \geqslant
\frac{\sqrt{2}}{3\pi}(2-\gamma_j)^{3/2}\bigg(1-\Big(1-\frac{\frac{u}{N}}{2-\gamma_j}\Big)^{3/2}\bigg)\\
 & \geqslant \frac{\sqrt{2}}{3\pi}(2-\gamma_j)^{1/2}\frac{u}{N},
\end{align*}
if $u \leqslant (2-\gamma_j)N$.
By definition, $1-\frac{j}{N}=\int_{\gamma_j}^2\frac{1}{2\pi}\sqrt{4-x^2}\,dx$. 
Then
 \[\frac{1}{\sqrt{2}\pi}\int_{\gamma_j}^2\sqrt{2-x}\,dx \leqslant 1-\frac{j}{N} \leqslant
 \frac{1}{\pi}\int_{\gamma_j}^2\sqrt{2-x}\,dx.\] Computing the preceding integrals yields
\begin{equation}\label{encadrement_gamma_j}
 \Big(\frac{3\pi}{2}\frac{N-j}{N}\Big)^{2/3} \leqslant 2-\gamma_j \leqslant
\Big(\frac{3\pi}{\sqrt{2}}\frac{N-j}{N}\Big)^{2/3}.
\end{equation}
Therefore, $u \leqslant (2-\gamma_j)N$ if $u \leqslant c'N$ with $c' \leqslant
(\frac{3\pi}{2})^{2/3}\eta^{2/3}$. In this case, \eqref{encadrement_gamma_j} yields
\[\rho_{\gamma_j+\frac{u}{N}}-\rho_{\gamma_j} \geqslant 2C\frac{u}{N} \]
where $C>0$. Therefore
\begin{align*}
 \P\Big(\lambda_j >\gamma_j+ \frac{u}{N}\Big) & \leqslant
\P\big(|\mathcal{N}_{\gamma_j+\frac{u}{N}}-N\rho_{\gamma_j+\frac{u}{N}}|>2Cu\big)\\
 & \leqslant \P\big(|\mathcal{N}_{\gamma_j+\frac{u}{N}}-N\rho_{\gamma_j+\frac{u}{N}}|>Cu+C_1\big)
\end{align*}
when $u \geqslant \frac{C_1}{C}=c$. Then,
applying
\eqref{inequality_counting_function} leads to
\[\P\Big(\lambda_j >\gamma_j+ \frac{u}{N}\Big) \leqslant 2\exp\Big(-\frac{C^2u^2
}{2\sigma^2_{\gamma_j+\frac{u}{N}}+Cu}\Big).\]
As $\gamma_j$ is in the bulk, there exist $\delta$ and $c'<(\frac{3\pi}{2})^{2/3}\eta^{2/3}$ such that
$\gamma_j+\frac{u}{N} \in I_{\delta}$, for all~$\eta N \leqslant j \leqslant (1-\eta)N$ and for all~$c
\leqslant u \leqslant c'N$, for all $N \geqslant 1$ (both $\delta$ and $c'$ depend on $\eta$). Then
\[\P\Big(\lambda_j >\gamma_j+ \frac{u}{N}\Big) \leqslant 2\exp\Big(-\frac{C^2u^2
}{2c_{\delta}\log N+Cu}\Big).\]
Repeating the argument leads to the same bound on $\P\big(\lambda_j <\gamma_j- \frac{u}{N}\big)$. Therefore,
\begin{equation*}
\P\Big(|\lambda_j-\gamma_j|\geqslant \frac{u}{N}\Big)\leqslant 4\exp\Big(-\frac{C^2u^2
}{2c_{\delta}\log N+Cu}\Big).
\end{equation*}
The proposition is thus established.

\end{proof}

\begin{proof}[Proof of Theorem \ref{thm_GUE_bulk}]
Note first that, for every $j$,
\begin{equation*}
\esp\big[\lambda_j^{4}\big] \leqslant \sum_{i=1}^N\esp\big[\lambda_i^{4}\big]=\esp\big[\Tr(W_N^4)\big].
\end{equation*}
The mean of this trace can be easily computed
and is equal to $2N+\frac{1}{N}$. Consequently, for all $N \geqslant 1$,
\begin{equation}\label{bound_moment}
\esp\big[\lambda_j^{4}\big] \leqslant 3N.
\end{equation}
Choose next $M=M(\eta)>0$ large enough such that
$\frac{C^2M^2}{2c_{\delta}+CM}>5$. Setting $Z=N|\lambda_j-\gamma_j|$,
\begin{align*}
\esp[Z^2] & = \int_0^{\infty}\P(Z\geqslant v)2v\,dv\\
 & = \int_0^{c}\P(Z\geqslant v)2v\,dv + \int_{c}^{M\log
N}\P(Z\geqslant v)2v\,dv + \int_{M\log N}^{\infty}\P(Z\geqslant v)2v\,dv\\
 & \leqslant c^2 + I_1 + I_2.
\end{align*}
The two latter integrals are handled in different ways. The first one $I_1$
is bounded  using \eqref{deviation_bulk} while $I_2$ is controlled using 
the Cauchy-Schwarz inequality and \eqref{bound_moment}.
Starting thus with $I_2$,
\begin{align*}
I_2 & = \int_{M\log N}^{+\infty}\P(Z\geqslant v)2v\,dv\\
 & \leqslant \esp\big[Z^2\mathbbm{1}_{Z \geqslant M\log N}\big]\\
 & \leqslant \sqrt{\esp\big[Z^{4}\big]} \sqrt{\P\big(Z \geqslant M\log N\big)}\\
 & \leqslant 2A
\exp\bigg(\frac{1}{2}\Big(5-\frac{C^2M^2}{2c_{\delta}+CM}\Big)\log N\bigg),
\end{align*}
where $A>0$ is a numerical constant. As
$\exp\Big(\frac{1}{2}\big(5-\frac{C^2M^2}{2c_{\delta}+CM}\big)\log N\Big)
\underset{N \to
\infty}{\to} 0$, there exists a constant $C(\eta)>0$ such that \[I_2 \leqslant
C(\eta).\]
Turning to $I_1$, recall that Proposition \ref{prop_eigenvalue_deviation_bulk} gives, for $c \leqslant v
\leqslant c'N$, 
\[P(Z\geqslant v) =
\P\Big(|\lambda_j-\gamma_j|\geqslant \frac{v}{N}\Big) \leqslant 4\exp\Big(-\frac{C^2v^2 }{2c_{\delta}\log
N+Cv}\Big).\] 
Hence in the range $v \leqslant M\log N$, 
\[P(Z\geqslant v) \leqslant
4\exp\Big(-\frac{B}{\log N}v^2\Big),\]
where $B=B(\eta)=\frac{C^2}{2c_{\delta}+CM}$.
There exists thus a constant $C(\eta)>0$ such that 
\[I_1 \leqslant C(\eta)\log N.\]
Summarizing the previous steps, $\esp\big[Z^2\big] \leqslant C(\eta)\log N$. Therefore
\[\esp\big[|\lambda_j-\gamma_j|^2\big]\leqslant
C(\eta)\frac{\log N}{N^2}, \]
which is the claim. The proof of Theorem~\ref {thm_GUE_bulk} is complete.
\end{proof}
It may be shown similarly that, under the assumptions of Theorem~\ref{thm_GUE_bulk},
\[\esp\big[|\lambda_j-\gamma_j|^p\big] \leqslant C(p,\eta)\frac{(\log N)^{p/2} }{N^p}.\]

\subsection{Eigenvalues at the edge of the spectrum}\label{edge_GUE}

In \cite{LeRi_2010_deviations}, Ledoux and Rider gave unified proofs of precise small deviation inequalities for the extreme
eigenvalues of $\beta$-ensembles. The results hold in particular for GUE matrices ($\beta=2$) and for GOE
matrices ($\beta=1$). The following theorem summarizes some of the relevant inequalities for the GUE.

\begin{thm}
There exists a universal constant $C>0$ such that the following holds. Let $M_N$ be a GUE matrix. Set $W_N=\frac{1}{\sqrt{N}}M_N$ and denote by $\lambda_N$ the maximal
eigenvalue of $W_N$. Then, for all $N \in \mathbb{N}$ and all $0< \varepsilon\leqslant 1$,
\begin{equation}
 \P\big(\lambda_{N} \geqslant 2(1+\varepsilon)\big) \leqslant C\exp\Big(-\frac{ 2N
}{C}\varepsilon^{3/2}\Big),
\end{equation}
and
\begin{equation}
\P\big(\lambda_{N} \leqslant 2(1-\varepsilon)\big) \leqslant C^2\exp\Big(-\frac{2 N^2
}{C}\varepsilon^3\Big). 
\end{equation}
\end{thm}
There exists also a right-tail large deviation inequality for $\varepsilon=O(1)$ of the form
\begin{equation}
\P\big(\lambda_{N} \geqslant 2(1+\varepsilon)\big) \leqslant C\exp\Big(-\frac{ 2N }{C}\varepsilon^2\Big),
\end{equation}
where $C>0$ is a universal constant.
Similar inequalities hold for the smallest eigenvalue $\lambda_{1}$. As stated in \cite{LeRi_2010_deviations}, bounds on the variance straightly follow from these deviation
inequalities.

\begin{cor}\label{thm_GUE_edge}
Let $M_N$ be a GUE matrix. Set $W_N=\frac{1}{\sqrt{N}}M_N$. Then there exists a universal
 constant $C>0$ such that for all $ N \geqslant 1$,
\[ \Var(\lambda_{N}) \leqslant \esp\big[(\lambda_N-2)^2\big] \leqslant  CN^{-4/3}.\]
\end{cor}
Similar results are probably true for the $k^{\textrm{th}}$ smallest or largest eigenvalue (for $k \in
\mathbb{N}$ fixed).

\subsection{Eigenvalues between the bulk and the edge of the spectrum}\label{intermediate_GUE}

The aim of this section is to establish a bound on the variance for some intermediate eigenvalues. The proof
is very similar to what was done for eigenvalues in the bulk. It relies on the fact that the eigenvalue
counting function has the same distribution as a sum of independent Bernoulli variables due to the
determinantal properties of the eigenvalues. A deviation inequality for individual eigenvalues can then be
derived and the bound on the variance straightly follows. Parts of the proof which are identical to the proof
for eigenvalues in the bulk will be omitted.
In what follows we only consider the right-side of the spectrum. Results and proofs for the left-side can be
deduced by replacing $N-j$ by $j$. The precise statement is the following.

\begin{thm}\label{thm_GUE_intermediate}
Let $M_N$ be a GUE matrix. Set $W_N=\frac{1}{\sqrt{N}}M_N$.
For all $K\geqslant 20\sqrt{2}$ and for all $\eta \in (0,\frac{1}{2}]$, there exists a constant $C(\eta,K)>0$
such that for all $N \geqslant 2$ and all $(1-\eta)N \leqslant j \leqslant N-K\log N $,
\begin{equation}
\esp\big[|\lambda_j-\gamma_j|^2\big] \leqslant C(\eta,K)\frac{\log (N-j)}{N^{4/3}(N-j)^{2/3}}.
\end{equation}
In particular,
\begin{equation}
\Var(\lambda_j) \leqslant C(\eta,K)\frac{\log (N-j)}{N^{4/3}(N-j)^{2/3}}.
\end{equation}
\end{thm}

The preceding theorem does not concern all intermediate eigenvalues
since $N-j$ has to be at least of the order of $20\sqrt{2}\log N$ for
the method used here to yield the correct order on the variance. This
restriction seems however to be technical
since Gustavsson \cite{Gu_2005_fluctuations} proved
that a Central Limit Theorem holds for all eigenvalues $\lambda_j$ such that $N-j \to \infty$ but $\frac{N-j}{N} \to 0$. From this CLT, the variance of such an individual eigenvalue $\lambda_j$ is guessed to be 
similarly of the order of $\frac{\log(N-j)}{N^{4/3}(N-j)^{2/3}}$ for this range.

As for eigenvalues in the bulk, the proof relies on the deviation inequality for the eigenvalue counting
function \eqref{inequality_counting_function}. From this and a bound on the variance of this counting
function \eqref{sup_variance}, it was then possible to derive a deviation inequality for eigenvalues in the
bulk. The work of Gustavsson \cite{Gu_2005_fluctuations}
suggests that for all $0<\tilde{\eta}<4$ and for
all $\tilde{K}>0$, there exists a constant $c_{\tilde{\eta},\tilde{K}}>0$ such that the following holds. For every sequence $(t_N)_{N \in \mathbb{N}}$ such that, for all $N$, $0<2-t_N \leqslant \tilde{\eta}$ and $N(2-t_N)^{3/2} \geqslant \tilde{K}\log N$,
\begin{equation}\label{sup_variance_intermediate}
\sigma_{t_N}^2 \leqslant c_{\tilde{\eta},\tilde{K}}\log\big(N(2-t_N)^{3/2}\big).
\end{equation}
Similarly to the bulk case, the following proposition can be established.

\begin{prop}[Deviation inequality for intermediate
$\lambda_j$]\label{prop_eigenvalue_deviation_intermediate}
There exist universal positive constants $C$ and $c$ such that the following holds.
Let $K > 20\sqrt{2}$ and $\eta \in (0,\frac{1}{2}]$. Set $(1-\eta)N \leqslant j \leqslant N-K\log N
$.
There exists $C'>0$ and $c'>0$ depending on $K$ and $\eta$ such
that for all
$c \leqslant u \leqslant c'(N-j)$, 
\begin{equation}\label{deviation_intermediate}
\P\Big(|\lambda_j-\gamma_j|\geqslant \frac{u}{N^{2/3}(N-j)^{1/3}}\Big)\leqslant 4\exp\Big(-\frac{C^2u^2
}{C'\log (N-j)+Cu}\Big).
\end{equation}
\end{prop}

\begin{proof}
Set $C=2^{-5/6}(3\pi)^{-2/3}$. 
Let $K>20\sqrt{2}$ and $\eta \in (0,\frac{1}{2}]$. Take $\alpha \in \big(\frac{20\sqrt{2}}{K},1\big)$ and set
$c'=\alpha\big(\frac{3\pi}{2}\big)^{2/3}$. For $(1-\eta)N \leqslant j \leqslant N-K\log N
$ and $u \geqslant 0$, set $u_{N,j}=\frac{u}{N^{2/3}(N-j)^{1/3}}$.
As in the proof of Proposition \ref{prop_eigenvalue_deviation_bulk}, evaluating the
probability~$\P(|\lambda_j-\gamma_j|>u_{N,j})$ yields
\[\P\Big(\lambda_j >\gamma_j+ u_{N,j}\Big) \leqslant
\P\big(|\mathcal{N}_{\gamma_j+u_{N,j}}-N\rho_{\gamma_j+u_{N,j}}
|>N(\rho_{ \gamma_j+u_{N,j}}-\rho_{
\gamma_j } )\big).\]
But, if $u \leqslant N^{2/3}(N-j)^{1/3}(2-\gamma_j)$,
\begin{align*}
\rho_{\gamma_j+u_{N,j}}-\rho_{\gamma_j} 
 & \geqslant \frac{\sqrt{2}}{3\pi}(2-\gamma_j)^{1/2}u_{N,j}.
\end{align*}
Similarly to the proof of Proposition \ref{prop_eigenvalue_deviation_bulk}, this condition holds if $u
\leqslant c'(N-j)$. In this case,
\[\P\Big(\lambda_j >\gamma_j+ u_{N,j}\Big) \leqslant
\P\big(|\mathcal{N}_{\gamma_j+u_{N,j}}-N\rho_{\gamma_j+u_{N,j}} |>Cu+C_1\big),\]
when $u \geqslant c=\frac{C_1}{C}$. Then, applying
\eqref{inequality_counting_function} leads to
\[ \P\Big(\lambda_j >\gamma_j+ u_{N,j}\Big) \leqslant
2\exp\Big(-\frac{C^2u^2
}{2\sigma^2_{\gamma_j+u_{N,j}}+Cu}\Big)\]
for all $c \leqslant u \leqslant c'(N-j)$.
Gustavsson's result \eqref{sup_variance_intermediate} gives a bound on $\sigma_{\gamma_j+u_{N,j}}^2$. Set
$t_N=\gamma_j+u_{N,j}$. As 
$j \geqslant (1-\eta)N$ and $u \geqslant 0$, $0 \leqslant 2-t_N \leqslant 2-\gamma_j \leqslant
\big(\frac{3\pi}{2}\eta\big)^{2/3}=\tilde{\eta}$ for
all $N$. Moreover,
\begin{align*}
N(2-t_N)^{3/2} & = N(2-\gamma_j-u_{N,j})^{3/2}\\
 & = N(2-\gamma_j)^{3/2}\Big(1-\frac{u_{N,j}}{2-\gamma_j}\Big)^{3/2}\\
  & \geqslant \frac{3\pi}{2}(N-j)\Big(1-\frac{c'(N-j)^{2/3}}{N^{2/3}(2-\gamma_j)}\Big)^{3/2}\\
 & \geqslant \frac{3\pi}{2}(1-\alpha)^{3/2}K\log N\\
 & \geqslant \tilde{K}\log N,
\end{align*}
where $\tilde{K}=\frac{3\pi}{2}(1-\alpha)^{3/2}K>0$.
From \eqref{sup_variance_intermediate}, for all $c \leqslant u \leqslant c'(N-j)$,
\[\Var(\mathcal{N}_{\gamma_j+ u_{N,j}}) 
    \leqslant c_{\tilde{\eta},\tilde{K}}\log\big(N(2-t_N)^{3/2}\big).\] But
\begin{align*}
N(2-t_N)^{3/2} & = N(2-\gamma_j-u_{N,j})^{3/2}\\
 & \leqslant N(2-\gamma_j)^{3/2}\\
 & \leqslant \frac{3\pi}{\sqrt{2}}(N-j).
\end{align*}
Hence $\log\big(N(2-t_N)^{3/2}\big) 
\leqslant \log(N-j)+\log(\frac{3\pi}{\sqrt{2}})$. As $K>20\sqrt{2}$ and $N
\geqslant 2$, $N-j \geqslant K\log N \geqslant \frac{3\pi}{\sqrt{2}}$ and $\Var(\mathcal{N}_{\gamma_j+
u_{N,j}}) \leqslant 2c_{\tilde{\eta},\tilde{K}}\log(N-j)$.
Therefore
\[\P\Big(\lambda_j >\gamma_j+ u_{N,j}\Big) \leqslant 2\exp\Big(-\frac{C^2u^2
}{4c_{\tilde{\eta},\tilde{K}}\log(N-j)+Cu}\Big).\]
The proof is concluded similarly to Proposition \ref{prop_eigenvalue_deviation_bulk}.
\end{proof}

On the basis of Proposition \ref {prop_eigenvalue_deviation_intermediate},
we may then conclude the proof of Theorem \ref{thm_GUE_intermediate}.

\begin{proof}[Proof of Theorem \ref{thm_GUE_intermediate}]
Setting $Z=N^{2/3}(N-j)^{1/3}|\lambda_j-\gamma_j|$,
\begin{align*}
\esp[Z^2] & = \int_0^{\infty}\P(Z\geqslant v)2v\,dv\\
 & = \int_0^{c}\P(Z\geqslant v)2v\,dv + \int_{c}^{\frac{C'}{C}\log(N-j)}\P(Z\geqslant v)2v\,dv\\
 & +
\int_{\frac{C'}{C}\log(N-j)}^{c'(N-j)}\P(Z\geqslant v)2v\,dv + \int_{c'(N-j)}^{\infty}\P(Z\geqslant
v)2v\,dv\\
 & \leqslant c^2 + J_1 + J_2+J_3.
\end{align*}
Repeating the computations carried out with $I_2$ in the proof of Theorem \ref{thm_GUE_bulk} yields
\begin{align*}
J_3  & \leqslant
2AN^{11/6}(N-j)^{2/3}\exp\Big(-\frac{1}{2}\frac{C^2c'^2(N-j)^2}{C'\log(N-j)+Cc'(N-j)}\Big),
\end{align*}
where $A>0$ is a numerical constant. For $N$ large enough (depending on $\eta$ and $K$), $C'\log(N-j)
\leqslant Cc'(N-j)$ and
\[ J_3 \leqslant
2AN^{5/2}\exp\Big(-\frac{Cc'}{4}(N-j)\Big).\]
Then, as $N-j \geqslant K\log N$,
\[J_3 \leqslant 2AN^{5/2}\exp\Big(-\frac{KCc'}{4}\log N\Big).\] But
$\frac{KCc'}{4}=\frac{K\alpha}{8\sqrt{2}} > \frac{5}{2}$. The right-hand side goes thus to $0$ when $N$
goes
to infinity. As a consequence, there exists a constant $C(\eta,K)>0$ such that \[J_3
\leqslant
C(\eta,K).\]
The integral $J_1$ is handled as $I_1$, using that, in the range $v \leqslant \frac{C'}{C}\log(N-j)$, 
 \[P(Z\geqslant v) \leqslant
 4\exp\Big(-\frac{B}{\log (N-j)}v^2\Big),\]
 where $B=B(\eta,K)=\frac{C^2}{2C'}$ (this is due to Proposition
 \ref{prop_eigenvalue_deviation_intermediate}).
Hence, there exists a constant $C(\eta,K)$ such that
\[J_1 \leqslant C(\eta,K)\log(N-j).\]
Finally, $J_2$ is handled similarly. From Proposition \ref{prop_eigenvalue_deviation_intermediate}, in the
range $\frac{C'}{C}\log(N-j) \leqslant v \leqslant c'(N-j)$,
\[P(Z\geqslant v) \leqslant 4\exp\Big(-\frac{C}{2}v\Big).\]
Thus
\[J_2 \leqslant 4\int_{\frac{C'}{C}\log(N-j)}^{c'(N-j)}\exp\Big(-\frac{C}{2}v\Big)2v\,dv \leqslant
4\int_{0}^{\infty}\exp\Big(-\frac{C}{2}v\Big)2v\,dv.\]
Then $J_2$ is bounded by a constant, which is independent of $\eta$ and $K$. 
There exists thus a constant $C>0$ such that 
\[J_2 \leqslant C.\]
Summarizing the previous steps, $\esp\big[Z^2\big] \leqslant C(\eta, K)\log (N-j)$. Therefore
\[\esp\big[|\lambda_j-\gamma_j|^2\big]\leqslant
C(\eta,K)\frac{\log (N-j)}{N^{2/3}(N-j)^{1/3}}, \]
which is the claim.
\end{proof}

\section{Variance bounds for Wigner Hermitian matrices}\label{Wigner}
As announced, the goal of this section is to prove Theorems \ref{thm_bulk}, \ref{thm_edge} and \ref{thm_intermediate} for Wigner Hermitian matrices. The eigenvalues of a Wigner Hermitian matrix do not form a determinantal process.
Therefore it does not seem easy to provide deviation inequalities for the counting function and for individual
eigenvalues. However the sharp non-asymptotic bounds established in the Gaussian case can still be reached by
a comparison procedure.

\subsection{Localization of the eigenvalues and the Four Moment Theorem}\label{tools_Wigner}
Two main recent theorems will be used in order to carry out this comparison procedure. First, Erdös, Yau and
Yin proved in \cite{ErYaYi_2010_rigidity} a Localization Theorem which gives a high probability non-asymptotic
bound on the distance between an eigenvalue $\lambda_j$ and its theoretical value $\gamma_j$. Secondly, Tao
and Vu's Four Moment Theorem (see \cite{TaVu_2011_wigner} and \cite{TaVu_2010_wigner_edge}) provides a very useful
non-asymptotic bound on the error made by approximating a statistics of the eigenvalues of a Wigner matrix by the
same statistics but with the eigenvalues of a GUE matrix.

Let $M_N$ be a Wigner Hermitian matrix. Say that $M_N$ satisfies condition $(C0)$ if the real part $\xi$ and
the imaginary part $\tilde{\xi}$ of $(M_N)_{ij}$ are independent and have an exponential decay: there are two positive constants $B_1$ and $B_2$ such that
\begin{equation*}
\P\big(|\xi|\geqslant t^{B_1}\big) \leqslant e^{-t} \quad \textrm{and} \quad \P\big(|\tilde{\xi}|\geqslant t^{B_1}\big)
\leqslant e^{-t}
\end{equation*}
for all $t \geqslant B_2$.

\begin{thm}[Localization \cite{ErYaYi_2010_rigidity}]\label{thm_localization}
Let $M_N$ be a random Hermitian matrix whose entries satisfy condition $(C0)$.
There are positive universal constants $c$ and $C$ such that, for any $ 1 \leqslant j \leqslant N $,
\begin{equation}\label{localization_inequality}
\P\Big(|\lambda_j-\gamma_j|\geqslant (\log N)^{C\log\log N}N^{-2/3}\min(j,N+1-j)^{-1/3}\Big)\leqslant
Ce^{-(\log
N)^{c\log\log N}}.
\end{equation}
\end{thm}
This a strong localization result and it almost yields the correct order on the bound on the variance. Indeed,
by means of the Cauchy-Schwarz inequality and \eqref{bound_moment}, it can be shown that, for $\eta N
\leqslant j
\leqslant (1-\eta)N$,
\[\Var(\lambda_j) \leqslant \esp\big[(\lambda_j-\gamma_j)^2\big] \leqslant \frac{C(\eta)}{N^2}(\log
N)^{C\log\log
N}.\]
In Tao and Vu's recent paper \cite{TaVu_2012_deviation} on deviation inequalities, the authors proved a more
precise localization result. They indeed established a bound similar to \eqref{localization_inequality} but
with $(\log N)^{A}$ instead of $(\log N)^{C\log\log N}$ (where $A>0$ is fixed). However this more precise
bound is of no help below, as the final bounds on the variances remain unchanged.

We turn now to Tao and Vu's Four Moment Theorem, in order to compare $W_N$ with a GUE matrix $W_N'$. Say that two complex random variables $\xi$ and $\xi'$ match to order $k$ if
\begin{equation*}
\esp\big[ \Re(\xi)^m \Im(\xi)^l\big]
    =\esp\big[\Re(\xi')^m \Im(\xi')^l\big]
\end{equation*}
for all $m, l \geqslant 0$ such that $m+l \leqslant k$.

\begin{thm}[Four Moment Theorem \cite{TaVu_2011_wigner}, \cite{TaVu_2010_wigner_edge}]\label{TV}

There exists a small positive constant $c_0$ such that the following holds.
Let $M_N=(\xi_{ij})_{1\leqslant i,j \leqslant N}$ and $M_N'=(\xi'_{ij})_{1\leqslant i,j \leqslant N}$ be
two random Wigner Hermitian matrices satisfying condition $(C0)$. Assume that, for $1\leqslant i<j\leqslant N$, $\xi_{ij}$ and $\xi'_{ij}$
match to order $4$ and that, for $1 \leqslant i \leqslant N$, $\xi_{ii}$ and $\xi'_{ii}$ match to order $2$.
Set $A_N=\sqrt{N}M_N$ and $A_N'=\sqrt{N}M_N'$.
Let $G: \mathbb{R} \to \mathbb{R}$ be a smooth fonction such that:
\begin{equation}\label{condition_derivees}
\forall \ 0 \leqslant k \leqslant 5, \quad \forall x \in \mathbb{R}, \quad \big|G^{(k)}(x)\big|
\leqslant
N^{c_0}.
\end{equation}
Then, for all $1 \leqslant i \leqslant N$ and for $N$ large
enough (depending on constants $B_1$ and $B_2$ in condition $(C0)$),
\begin{equation}
\big|\esp\big[G(\lambda_{i}(A_N))\big]-\esp\big[G(\lambda_{i}(A_N'))\big]\big|\leqslant N^{-c_0}.
\end{equation}
\end{thm}
Actually Tao and Vu proved this theorem in a more general form, involving a finite number of eigenvalues. In
this work, it will only be used with one given eigenvalue. See \cite{TaVu_2011_wigner} and
\cite{TaVu_2010_wigner_edge} for more
details.

It should be mentionned that Tao and Vu extended in \cite{TaVu_2011_wigner} Gustavsson's result (see equation
\eqref{gustavsson_bulk}) via this theorem. By means of a smooth bump function $G$, they compared the
probability for $\lambda_j$
to be in a given interval for a non-Gaussian matrix with almost the same probability but for a GUE matrix.
Applying this technique to $\P\big(|\lambda_j-\gamma_j| >\frac{u}{N}\big)$ in order to extend directly the
deviation inequality leads to the following in the general Wigner case: for all $\eta N \leqslant j \leqslant
(1-\eta)N$ and for all $C' \leqslant u \leqslant c'N$, \[\P\Big(|\lambda_j-\gamma_j| >\frac{u}{N}\Big)
\leqslant C\exp\Big(-\frac{u^2}{c\log N+u}\Big)+O(N^{-c_0}),\]
where $C$, $C'$, $c$ and $c'$ are positive constants depending only on $\eta$.
The bound is not exponential anymore and is
not enough to conclude towards sharp bounds on the variance or higher moments.

\subsection{Comparison with Gaussian matrices}\label{comparison}
Let $M_N$ be a Hermitian Wigner matrix and $M_N'$ be a GUE matrix such that they satisfy the hypotheses of
Theorem \ref{TV}. As the function $G:x \in \mathbb{R} \mapsto x^2$ does not satisfy
\eqref{condition_derivees}, Theorem \ref{TV} will be applied to a truncation of $G$. Theorem
\ref{thm_localization} will provide a small area around the theoretical location $\gamma_j$ where the
eigenvalue $\lambda_j$ is very likely to be in, so that the error due to the truncation will be well
controlled. Note that this procedure is valid for eigenvalues in the bulk and at the edge of the
spectrum, as well as for intermediate eigenvalues.

Let $1 \leqslant j \leqslant N$.
Set $R_N^{(j)}=(\log N)^{C\log\log
 N}N^{1/3}\min(j,N+1-j)^{-1/3}$ and $\varepsilon_N=Ce^{-(\log N)^{c\log\log N}}$. Then Theorem
\ref{thm_localization} leads to:
\begin{equation}\label{short_localization_inequality}
\P\Big(|\lambda_j-\gamma_j| \geqslant \frac{R_N^{(j)}}{N}\Big) \leqslant \varepsilon_N. 
\end{equation}
Let $\psi$ be a smooth function with support $[-2,2]$ and values in $[0,1]$ such that
$\psi(x)=\frac{1}{10}x^2$ for all $x \in [-1; 1]$. Set $G_j: x \in \mathbb{R} \mapsto
\psi\big(\frac{x-N\gamma_j}{R_N^{(j)}}\big)$. We want to apply Tao and Vu's Four Moment Theorem \ref{TV} to $G_j$. As $\psi$ is smooth and has compact support, its first five derivatives are bounded by $M>0$. Then, for all $0 \leqslant k \leqslant 5$, for all $x \in \mathbb{R}$,
\begin{equation*}
\big|G_j^{(k)}(x)\big|\leqslant \frac{M}{(R_N^{(j)})^k} \leqslant N^{c_0},
\end{equation*}
where the last inequality holds for $N$ large enough (depending only on $M$ and $c_0$). Then, the Four Moment
Theorem \ref{TV} yields: 
\begin{equation}\label{application_four_moment_thm}
\big|\esp\big[G_j(\lambda_j(A_N))\big]-\esp\big[G_j(\lambda_j(A_N'))\big]\big| \leqslant N^{-c_0}
\end{equation}
for large enough $N$. But 
\begin{align*}
\esp\big[G_j(&\lambda_j(A_N))\big] \\
& = \tfrac{1}{10}\esp\Big[\Big(\tfrac{\lambda_j(A_N)-N\gamma_j}{R_N^{(j)}}\Big)^2\mathbbm{1}_{\frac{
|\lambda_j(A_N)-N\gamma_j|}{R_N^{(j)}}\leqslant 1}\Big]
+\esp\Big[G_j(\lambda_j(A_N))\mathbbm{1}_{\frac{
|\lambda_j(A_N)-N\gamma_j|}{R_N^{(j)}}> 1}\Big]\\
 & = \tfrac{N^2}{10\big(R_N^{(j)}\big)^2}\esp\Big[(\lambda_j-\gamma_j)^2\mathbbm{1}_{|\lambda_j-\gamma_j|\leqslant \frac{R_N^{(j)}}{N}}\Big]+\esp\Big[G_j(\lambda_j(A_N))\mathbbm{1}_{\frac{
|\lambda_j(A_N)-N\gamma_j|}{R_N^{(j)}}> 1}\Big].\\
\end{align*}
On the one hand,
\begin{equation*}
 \esp\Big[G_j(\lambda_j(A_N))\mathbbm{1}_{\frac{
|\lambda_j-N\gamma_j|}{R_N^{(j)}}> 1}\Big]  \leqslant 
\P\Big(|\lambda_j(W_N)-\gamma_j|>\tfrac{R_N^{(j)}}{N}\Big) \leqslant  \varepsilon_N.
\end{equation*}
On the other hand,
\begin{equation*}
\esp\Big[(\lambda_j-\gamma_j)^2\mathbbm{1}_{
|\lambda_j-\gamma_j|\leqslant\frac{R_N^{(j)}}{N}}\Big] = \esp\big[(\lambda_j-\gamma_j)^2\big] -
\esp\Big[(\lambda_j-\gamma_j)^2\mathbbm{1}_{|\lambda_j-\gamma_j|>\frac{R_N^{(j)}}{N}}\Big].
\end{equation*}
From the Cauchy-Schwarz inequality and \eqref{bound_moment},
\begin{eqnarray*}
\esp\Big[(\lambda_j-\gamma_j)^2\mathbbm{1}_{|\lambda_j-\gamma_j|>\frac{R_N^{(j)}}{N}}\Big] & \leqslant &
\sqrt{\esp\big[(\lambda_j-\gamma_j)^4\big]\P\Big(|\lambda_j-\gamma_j|>\tfrac{R_N^{(j)}}{N}\Big)}\\
 & \leqslant &  A\sqrt{N\varepsilon_N}
\end{eqnarray*}
where $A>0$ is a numerical constant. Then
\begin{eqnarray*}
\esp\big[G_j(\lambda_j(A_N))\big] & = &
\frac{N^2}{10\big(R_N^{(j)}\big)^2}\Big(\esp\big[(\lambda_j-\gamma_j)^2\big]+O\big(N^{1/2}\varepsilon_N^{1/2}
\big)\Big)+O(\varepsilon_N)\\
 & = &
\frac{N^2}{10\big(R_N^{(j)}\big)^2}\esp\big[(\lambda_j-\gamma_j)^2\big]+O\big(N^{5/2}\varepsilon_N^{1/2}\big(R_N^{(j)}\big)^{-2}
\big)+O(\varepsilon_N).\\
\end{eqnarray*}
Repeating the same computations gives similarly
\begin{equation*}
\esp\big[G_j(\lambda_j(A_N'))\big] = 
\frac{N^2}{10\big(R_N^{(j)}\big)^2}\esp\big[(\lambda_j'-\gamma_j)^2\big]+O\big(N^{5/2}\varepsilon_N^{1/2}\big(R_N^{(j)}\big)^{-2}
\big)+O(\varepsilon_N).\\
\end{equation*}
Then \eqref{application_four_moment_thm} leads to
\[\esp\big[(\lambda_j-\gamma_j)^2\big]=\esp\big[(\lambda_j'-\gamma_j)^2\big]+O\big(N^{1/2}\varepsilon_N^{1/2}\big)+ 
O\big(N^{-2}\big(R_N^{(j)}\big)^2\varepsilon_N\big)+O\big(\big(R_N^{(j)}\big)^2N^{-c_0-2}\big). \]
As the first two error terms are smaller than the third one, the preceding equation becomes 
\begin{equation}\label{difference_second_moment}
\esp\big[(\lambda_j-\gamma_j)^2\big]=\esp\big[(\lambda_j'-\gamma_j)^2\big]+O\big(\big(R_N^{(j)}\big)^2N^{-c_0-2}\big).
\end{equation}

\subsection{Combining the results}\label{results_wigner}
We distinguish between the bulk, the edge and the intermediate cases.
Note that the constants $C(\eta)$ and $C$ depend on
the constants $B_1$ and $B_2$
in condition $(C0)$.

\subsubsection{Eigenvalues in the bulk of the spectrum}\label{bulk_wigner}
Let $0< \eta \leqslant \frac{1}{2}$ and $\eta N\leqslant j \leqslant (1-\eta)N$. From Theorem
\ref{thm_GUE_bulk},
$\esp\big[(\lambda_j'-\gamma_j)^2\big] \leqslant C(\eta)\frac{\log N}{N^2}$. Thus, from
\eqref{difference_second_moment}, it remains to show that the error term is smaller than $\frac{\log N}{N^2}$.
But
\[R_N^{(j)}=(\log N)^{C\log\log N}N^{1/3}\min(j,N+1-j)^{-1/3} \leqslant \eta^{-1/3}(\log
N)^{C\log\log N}.\]
Then $\big(R_N^{(j)}\big)^2N^{-c_0-2}=o_{\eta}\big(\frac{\log N}{N^2}\big)$. As a consequence, \[
\esp\big[(\lambda_j-\gamma_j)^2\big]=\esp\big[(\lambda_j'-\gamma_j)^2\big]+o_{\eta}\Big(\frac{\log
N}{N^2}\Big)
\]
and we get the desired result
\[
\esp\big[(\lambda_j-\gamma_j)^2\big]\leqslant C(\eta)\frac{\log N}{N^2}.
\]

\subsubsection{Eigenvalues at the edge of the spectrum}\label{edge_wigner}
From Corollary \ref{thm_GUE_edge}, $\esp\big[(\lambda_N'-\gamma_N)^2\big]=
\esp\big[(\lambda_N'-2)^2\big]\leqslant CN^{-4/3}$. By means
of \eqref{difference_second_moment}, it remains to prove that the error term is smaller than $N^{-4/3}$.
We have \[R_N^{(N)}= (\log N)^{C\log\log N}N^{1/3}.\]
Consequently
$\big(R_N^{(N)}\big)^2N^{-c_0-2}=o\big(N^{-4/3}\big)$. Then \[
\esp\big[(\lambda_N-2)^2\big]=\esp\big[(\lambda_N'-2)^2\big]+o\big(N^{-4/3}\big)
\]
and
\[
\esp\big[(\lambda_N-2)^2\big]\leqslant CN^{-4/3}.
\]
As for Gaussian matrices, the same result is available for the smallest eigenvalue $\lambda_1$.

\subsubsection{Eigenvalues between the bulk and the edge of the spectrum}\label{intermediate_Wigner}
Let $0< \eta \leqslant \frac{1}{2}$, $K > 20\sqrt{2}$ and $(1-\eta)N \leqslant j \leqslant N-K\log N$. From Theorem
\ref{thm_GUE_intermediate},
$\esp\big[(\lambda_j'-\gamma_j)^2\big] \leqslant C(\eta,K)\frac{\log (N-j)}{N^{4/3}(N-j)^{2/3}}$. Thus, from
\eqref{difference_second_moment}, it remains to show that the error term is smaller than $\frac{\log
(N-j)}{N^{4/3}(N-j)^{2/3}}$.
But
\[R_N^{(j)}=(\log N)^{C\log\log N}N^{1/3}(N+1-j)^{-1/3}.\]
Then $\big(R_N^{(j)}\big)^2N^{-c_0-2}=o\big(\frac{\log (N-j)}{N^{4/3}(N-j)^{2/3}}\big)$. As a
consequence, \[
\esp\big[(\lambda_j-\gamma_j)^2\big]=\esp\big[(\lambda_j'-\gamma_j)^2\big]+o\Big(\frac{\log
(N-j)}{N^{4/3}(N-j)^{2/3}}\Big)
\]
and we get the desired result
\[
\esp\big[(\lambda_j-\gamma_j)^2\big]\leqslant C(\eta,K)\frac{\log
(N-j)}{N^{4/3}(N-j)^{2/3}}.\]
A similar result holds for the left-side of the spectrum.

\section{Real matrices}\label{real}
The goal of this section is to prove Theorems \ref{thm_bulk}, \ref{thm_edge} and \ref{thm_intermediate} for
real Wigner matrices. Tao and Vu's Four Moment Theorem (Theorem \ref{TV}) as well as Erd\"os, Yau and Yin's
Localization Theorem (Theorem \ref{thm_localization}) still hold for real Wigner matrices. Section
\ref{Wigner} is therefore valid for real matrices. The point is then to establish the results in the GOE case.

As announced in Section \ref{edge_GUE}, the variance of eigenvalues at the edge of the spectrum is known to be
bounded by $N^{-4/3}$ for GOE matrices (see \cite{LeRi_2010_deviations}). The conclusion for the smallest
and largest eigenvalues is then established for large families of real symmetric Wigner matrices.
\[\Var(\lambda_N) \leqslant \frac{\tilde{C}}{N^{4/3}} \quad \textrm{and} \quad \Var(\lambda_1)\leqslant
\frac{\tilde{C}}{N^{4/3}}.\]

For eigenvalues in the bulk of the spectrum, O'Rourke proved in \cite{Or_2010_fluctuations} a Central Limit
Theorem which is very similar to the one established by Gustavsson in \cite{Gu_2005_fluctuations}. In
particular, the normalisation is still of the order of $\big(\frac{\log N}{N^2}\big)^{1/2}$ and differs from
the
complex
case only by a constant. It is therefore natural to expect the same bound on the variance for GOE matrices.
The situation is completely similar for intermediate eigenvalues.
But GOE matrices do not have the same determinantal properties as GUE matrices, and it is therefore not clear
that a deviation inequality (similar to \eqref{inequality_counting_function}) holds for the eigenvalue
counting function. However, as explained by O'Rourke in \cite{Or_2010_fluctuations}, GOE and GUE matrices are
linked by interlacing formulas established by Forrester and Rains (see \cite{FoRa_2001_relationships}). These
formulas lead to the following relation between the eigenvalue counting functions in the complex and in the
real cases: for all $t \in
\mathbb{R}$,

\begin{equation}\label{entrelacement}
\mathcal{N}_{t}(W_N^{\mathbb{C}})=\frac{1}{2}\big(\mathcal{N}_{t}(W_N^{\mathbb{R}})+\mathcal{N}_{t}(\tilde{W}_N^{\mathbb{R}}
)\big)+\zeta_N(t) ,
\end{equation}
where
$W_N^{\mathbb{C}}=\frac{1}{\sqrt{N}}M_N^{\mathbb{C}}$ and $M_N^{\mathbb{C}}$ is from the GUE, $W_N^{\mathbb{R}}=\frac{1}{\sqrt{N}}M_N^{\mathbb{R}}$,
$\tilde{W}_N^{\mathbb{R}}=\frac{1}{\sqrt{N}}\tilde{M}_N^{\mathbb{R}}$ and $M_N^{\mathbb{R}}$ and
$\tilde{M}_N^{\mathbb{R}}$ are independent matrices from the GOE and $\zeta_N(t)$ takes values in $\left\{-1,-\frac{1}{2},0,\frac{1}{2},1\right\}$.
See \cite{Or_2010_fluctuations} for more details.

The aim is now to establish a deviation inequality for the eigenvalue counting function similar to \eqref{inequality_counting_function}. From \eqref{inequality_counting_function}, we know that for all $u \geqslant 0$, 
\begin{equation*}
 \P\big(|\mathcal{N}_t(W_N^{\mathbb{C}})-N\rho_t|\geqslant u+C_1\big)\leqslant 2\exp\Big(-\frac{u^2}{2\sigma_t^2+u}\Big).
\end{equation*}
Set $C_1'=C_1+1$ and let $u \geqslant 0$. We can then write
\begin{align*}
\P\big(\mathcal{N}_t(W_N^{\mathbb{R}})- N & \rho_t \geqslant u+C_1'\big)^2 \\
& =  \P\Big(
\mathcal{N}_t(W_N^{\mathbb{R}})-N\rho_t \geqslant u+C_1', \ \mathcal{N}_t(\tilde{W}_N^{\mathbb{R}})-N\rho_t \geqslant u+C_1'\Big)\\
 & \leqslant  \P\Big(\tfrac{1}{2}\big(\mathcal{N}_t(W_N^{\mathbb{R}})+\mathcal{N}_t(\tilde{W}_N^{\mathbb{R}})\big)-N\rho_t \geqslant u+C_1'\Big)\\
 & \leqslant  \P\big(\mathcal{N}_t(W_N^{\mathbb{C}})-N\rho_t \geqslant u+C_1'-1\big)\\
 & \leqslant  2\exp\Big(-\frac{u^2}{2\sigma_t^2+u}\Big).
\end{align*}
Repeating the computations for $\P\big(\mathcal{N}_t(W_N^{\mathbb{R}})-N\rho_t \leqslant -u-C_1'\big)$ and
combining with the preceding yield
\begin{equation}\label{real_inequality_counting_function}
 \P\big(|\mathcal{N}_t(W_N^{\mathbb{R}})-N\rho_t| \geqslant u+C_1'\big) \leqslant 2\sqrt{2}\exp\Big(-\frac{u^2}{4\sigma_t^2+2u}\Big).
\end{equation}
Note that $\sigma_t^2$ is still the variance of $\mathcal{N}_t(W_N^{\mathbb{C}})$ in the preceding formula.

What remains then to be proved is very similar to the complex case. From
\eqref{real_inequality_counting_function} and Gustavsson's bounds on the variance $\sigma_t^2$ (see
\eqref{sup_variance} for the bulk case and \eqref{sup_variance_intermediate} for the intermediate
case), deviation inequalities for individual eigenvalues can be deduced, as was done to prove
Propositions~\ref{prop_eigenvalue_deviation_bulk} and \ref{prop_eigenvalue_deviation_intermediate}.
It is
then straightforward to derive the announced bounds on the variances for GOE matrices.
The argument developed in Section \ref{Wigner} in order to extend the GUE results to large families of
Hermitian Wigner matrices can be reproduced to reach the desired bounds on the variances of eigenvalues in
the bulk and between the bulk and the edge of the spectrum for families of real Wigner matrices.
Then there exists a constant $C(\eta)>0$ such that for all $\eta N \leqslant
j \leqslant (1-\eta)N$,
\[\Var(\lambda_j) \leqslant \esp\big[(\lambda_j-\gamma_j)^2\big] \leqslant C(\eta)\frac{\log N}{N^2},\]
and there exists a constant $C(\eta,K)>0$ such that for all $(1-\eta)N \leqslant j \leqslant N-K\log N$ (and
similarly for the left-side of the spectrum),
\[\Var(\lambda_j) \leqslant \esp\big[(\lambda_j-\gamma_j)^2\big] \leqslant C(\eta,K)\frac{\log
(N-j)}{N^{4/3}(N-j)^{2/3}}.\]

\section{A corollary on the $2$-Wasserstein distance}\label{wasserstein}

The bounds on the variances, more exactly on $\esp[(\lambda_j-\gamma_j)^2]$,
developed in the preceding sections lead to a bound on the rate of
convergence of the empirical spectral measure $L_N$ towards the semicircle law $\rho_{sc}$ in terms of
$2$-Wasserstein distance. Recall that $W_2(L_N,\rho_{sc})$ is a random variable defined by
\[W_2(L_N,\rho_{sc})
    =\inf \bigg (\int_{\mathbb{R}^2}|x-y|^2\,d\pi(x,y)\bigg)^{1/2},\]
where the infimum is taken over all probability measures $\pi$ on $\mathbb{R}^2$ with respective
marginals $L_N$ and $\rho_{sc}$. To achieve the expected bound, we rely
on another expression of $W_2$ in terms of distribution functions, namely
\begin{equation}\label{w2_fct_repartition}
 W_2^2(L_N,\rho_{sc}) = \int_0^1\big(F_N^{-1}(x)-G^{-1}(x)\big)^2\,dx,
\end{equation}
where $F_N^{-1}$ (respectively $G^{-1}$) is the generalized inverse of the distribution function $F_N$
(respectively $G$) of $L_N$ (respectively $\rho_{sc}$) (see for example \cite{Vi_2003_book}). 
On the basis of this representation, the following statement may be derived.

\begin{prop}\label{prop_w2}
There exists a universal constant $C>0$ such that for all $N \geqslant 1$,
\begin{equation}\label{borne_ps_w2}
 W_2^2(L_N,\rho_{sc}) \leqslant \frac{2}{N}\sum_{j=1}^N
(\lambda_j-\gamma_j)^2+\frac{C}{N^2}.
\end{equation}
\end{prop}
\begin{proof} From \eqref{w2_fct_repartition},
\[W_2^2(L_N,\rho_{sc}) = \int_0^1\big(F_N^{-1}(x)-G^{-1}(x)\big)^2\,dx.\] Then,
\begin{align*}
 W_2^2(L_N,\rho_{sc}) & = \sum_{j=1}^N\int_{\frac{j-1}{N}}^{\frac{j}{N}}\big(\lambda_j-G^{-1}(x)\big)^2\,dx\\
 & \leqslant \frac{2}{N}\sum_{j=1}^N(\lambda_j-\gamma_j)^2 +
2\sum_{j=1}^N\int_{\frac{j-1}{N}}^{\frac{j}{N}}\big(\gamma_j-G^{-1}(x)\big)^2\,dx.
\end{align*}
But $\gamma_j=G^{-1}\big(\frac{j}{N}\big)$ and $G^{-1}$ is non-decreasing. Therefore, $\big|\gamma_j-G^{-1}(x)\big| \leqslant
\gamma_{j}-\gamma_{j-1}$ for all $x \in \big[\frac{j-1}{N},\frac{j}{N}\big]$. Consequently,
\begin{equation}\label{inegalite_w2_intermediaire}
W_2^2(L_N,\rho_{sc}) \leqslant \frac{2}{N}\sum_{j=1}^N(\lambda_j-\gamma_j)^2 +
\frac{2}{N}\sum_{j=1}^N(\gamma_j-\gamma_{j-1})^2.
\end{equation}
But if $j-1 \geqslant \frac{N}{2}$ (and therefore $\gamma_{j-1} \geqslant 0$),
\begin{align*}
\frac{1}{N} & = \int_{\gamma_{j-1}}^{\gamma_j}\frac{1}{2\pi}\sqrt{4-x^2}\,dx\\
  & \geqslant \frac{1}{\sqrt{2}\pi}\int_{\gamma_{j-1}}^{\gamma_j}\sqrt{2-x}\,dx\\
 & \geqslant
\frac{\sqrt{2}}{3\pi}(2-\gamma_{j-1})^{3/2}\bigg(1-\Big(1-\frac{\gamma_j-\gamma_{j-1}}{2-\gamma_{j-1}}\Big)^{
3/2 }\bigg)\\
  & \geqslant
 \frac{\sqrt{2}}{3\pi}(2-\gamma_{j-1})^{1/2}(\gamma_j-\gamma_{j-1})\\
 & \geqslant
\frac{\sqrt{2}}{3\pi}\Big(\frac{3\pi}{2}\frac{N-j+1}{N}\Big)^{1/3}(\gamma_j-\gamma_{j-1}),
\end{align*}
from \eqref{encadrement_gamma_j}. Then
\[\gamma_j-\gamma_{j-1} \leqslant \frac{(3\pi)^{2/3}2^{-1/6}}{N^{2/3}(N-j+1)^{2/3}}.\]
It may be shown that a similar bound holds if $j-1 \leqslant \frac{N}{2}$. As a summary, there exists a
universal constant $c>0$ such that, for all $j \geqslant
2$,
\begin{equation}
\gamma_j-\gamma_{j-1} \leqslant \frac{c}{N^{2/3}\min(j,N+1-j)^{1/3}}.
\end{equation}
This yields
\begin{equation*}
 \sum_{j=1}^N(\gamma_j-\gamma_{j-1})^2  \leqslant 
\frac{c^2}{N^{4/3}}\sum_{j=1}^N\frac{1}{\min(j,N+1-j)^{2/3}}\leqslant \frac{C}{N}, 
\end{equation*}
where $C>0$ is a universal constant.
Then \eqref{inegalite_w2_intermediaire} becomes
\[W_2^2(L_N,\rho_{sc}) \leqslant \frac{2}{N}\sum_{j=1}^N
(\lambda_j-\gamma_j)^2+\frac{C}{N^2},\]
where $C>0$ is a universal constant, which is the claim.
\end{proof}

\begin{proof}[Proof of Corollary \ref{cor_wasserstein_2}]
Let $N \geqslant 2$. Due to Proposition \ref{prop_w2},
\[\esp\big[W_2^2(L_N,\rho_{sc})\big] \leqslant \frac{2}{N}\sum_{j=1}^N
\esp\big[(\lambda_j-\gamma_j)^2\big]+\frac{C}{N^2}.\]
We then make use of the bounds on $\esp\big[(\lambda_j-\gamma_j)^2\big]$
produced in the previous sections.
Set $\eta \in (0,\frac{1}{2}]$ and $K>20\sqrt{2}$ so that $K\log N \leqslant \eta N$.
We first decompose
\begin{align*}
\sum_{j=1}^N \esp\big[(\lambda_j-\gamma_j)^2\big] & = \sum_{j=1}^{K\log
N}\esp\big[(\lambda_j-\gamma_j)^2\big] + \sum_{j=K\log N+1}^{\eta
N}\esp\big[(\lambda_j-\gamma_j)^2\big]\\
 & + \sum_{j=\eta N+1}^{(1-\eta)N-1}\esp\big[(\lambda_j-\gamma_j)^2\big] + \sum_{j=(1-\eta)N}^{N-K\log
N-1}\esp\big[(\lambda_j-\gamma_j)^2\big]\\
 & +\sum_{j=N-K\log N}^{N}\esp\big[(\lambda_j-\gamma_j)^2\big]\\
 & = \Sigma_1 + \Sigma_2 + \Sigma_3 + \Sigma_4 + \Sigma_5.
\end{align*}
The sum $\Sigma_3$ will be bounded using the bulk case (Theorem \ref{thm_bulk}), while Theorem \ref{thm_edge} will be
used to handle $\Sigma_2$ and $\Sigma_4$. 
A crude version of Theorem \ref{thm_localization} will be enough to
bound $\Sigma_1$ and $\Sigma_5$. To start with thus, from Theorem \ref{thm_bulk},
\[\Sigma_3 \leqslant \sum_{j=\eta N+1}^{(1-\eta)N-1}C(\eta)\frac{\log N}{N^2} \leqslant C(\eta)\frac{\log
N}{N}.\]
Secondly, from Theorem \ref{thm_intermediate},
\begin{equation*}
\Sigma_2+\Sigma_4  \ \leqslant \ \frac{C(\eta,K)}{N^{4/3}}\sum_{j=K\log N+1}^{\eta
N}\frac{\log j}{j^{2/3}} \ \leqslant \ C(\eta,K)\frac{\log N}{N}.
\end{equation*}
Next $\Sigma_1$ and $\Sigma_5$ have only $K\log N$ terms. If each term is bounded by $\frac{C}{N}$ where
$C$ is a positive universal constant, we get that $\Sigma_1+\Sigma_5 \leqslant \frac{2KC\log N}{N}$,
which is enough to prove the desired result on $\sum_{j=1}^N\esp\big[(\lambda_j-\gamma_j)^2\big]$. For $N$
large
enough depending only on constant $C$ in Theorem \ref{thm_localization}, $\frac{1}{\sqrt{N}} \geqslant
\frac{(\log N)^{C\log\log N}}{N^{2/3}\min(j,N+1-j)^{1/3}}$ and Theorem \ref{thm_localization} yields
\[\P\Big(|\lambda_j-\gamma_j| \geqslant \frac{1}{\sqrt{N}}\Big) \leqslant
Ce^{-(\log N)^{c\log \log N}}.\] Then, by the Cauchy-Schwarz inequality,
\begin{align*}
 \esp\big[(\lambda_j-\gamma_j)^2\big] & \leqslant
\esp\Big[(\lambda_j-\gamma_j)^2\mathbbm{1}_{|\lambda_j-\gamma_j|\leqslant \frac{1}{\sqrt{N}}}\Big]+
\esp\Big[(\lambda_j-\gamma_j)^2\mathbbm{1}_{|\lambda_j-\gamma_j|> \frac{1}{\sqrt{N}}}\Big]\\
 & \leqslant
\frac{1}{N}+\sqrt{\esp\big[|\lambda_j-\gamma_j|^4\big]}\sqrt{\P\Big(|\lambda_j-\gamma_j| > \frac{1}{\sqrt{N}}\Big)}\\
& \leqslant
\frac{1}{N}+\sqrt{3}C N^{1/2}e^{-(\log N)^{c\log \log N}}.
\end{align*}
As $\sqrt{3}C N^{1/2}e^{-(\log N)^{c\log \log N}}=o(\frac{1}{N})$, there exists a constant $C>0$ such that
$\esp\big[(\lambda_j-\gamma_j)^2\big] \leqslant \frac{C}{N}$.
Then
\[\Sigma_1+\Sigma_5 \leqslant 2KC\frac{\log N}{N}.\]
As a consequence,
\[\sum_{j=1}^N \esp\big[(\lambda_j-\gamma_j)^2\big] \leqslant C\frac{\log N}{N}.\]
Therefore
\[\esp\big[W_2(L_N,\rho_{sc})^2\big] \leqslant C\frac{\log N}{N^2},\]
where $C>0$ is a universal constant, which is the claim. The corollary is thus established.
\end{proof}

The preceding Corollary \ref{cor_wasserstein_2} implies that
\[\esp\big[W_1(L_N,\rho_{sc})\big] \leqslant C\frac{\sqrt{\log N}}{N}.\]
This is an improvement of Meckes and Meckes' rate of convergence obtained in \cite{MeMe_2011_convergence_spectral_measure}. Note however that the distance
studied in \cite{MeMe_2011_convergence_spectral_measure}
is the expected $1$-Wasserstein distance between $L_N$ and its mean instead of $\rho_{sc}$.
The rate of convergence in $1$-Wasserstein distance can be  furthermore
compared to the rate of convergence in Kolmogorov distance. Indeed, if $\mu$ and $\nu$ are two probability measures on $\mathbb{R}$ such that $\nu$ has a bounded density function with respect to Lebesgue measure, $d_K(\mu,\nu) \leqslant c\sqrt{W_1(\mu,\nu)}$, where $c>0$ depends on the bound on the density function. This is due to the fact that
\[ W_1(\mu,\nu) = \sup_{f  1\textrm{-Lipschitz}}\bigg(\int_{\mathbb{R}} fd\mu-\int_{\mathbb{R}} fd\nu\bigg) \]
(see for example \cite{Vi_2003_book}) and $d_K(\mu,\nu)=\sup_{x \in \mathbb{R}} \big|\mu\big((-\infty,x]\big)-\nu\big((-\infty,x]\big)\big|$. Approximating $\mathbbm{1}_{(-\infty,x]}$ from above and from below by $\frac{1}{\varepsilon}$-Lipschitz functions and optimizing on $\varepsilon$ gives the result. Therefore,
the preceding implies that
\[\esp\big[d_K(L_N,\rho_{sc})\big] \leqslant c\frac{(\log N)^{1/4}}{N^{1/2}} \]
which is however far from Götze and Tikhomirov's recent bound \cite{GoTi_2011_convergence}
\[\esp\big[d_K(L_N,\rho_{sc})\big] \leqslant \frac{(\log N)^c}{N}. \]

\section{Eigenvalue variance bounds for covariance matrices}\label{covariance}
This section provides the analogous non-asymptotic bounds on the variance of eigenvalues for covariance
matrices. The proofs will be detailed in another redaction. Therefore, this section contains only the
background and the results. 

Random covariance matrices are defined by the following. Let $X$ be a $m\times n$ (real or complex) matrix,
with $m
\geqslant n$, such that its entries are independent, centered and have variance $1$. Then
$S_{m,n}=\frac{1}{n}X^*X$ is a random covariance matrix. An important example is the case when the entries of
$X$ are
Gaussian. Then $S_{m,n}$ belongs to the so-called Laguerre Unitary Ensemble (LUE) if the entries of $X$ are
complex and to the Laguerre Orthogonal Ensemble (LOE) if they are real. $S_{m,n}$ is Hermitian (or real
symmetric) and therefore has $n$ real eigenvalues. As $m \geqslant n$, none of these eigenvalues is trivial.
Furthermore, these eigenvalues are nonnegative and will be denoted by $0 \leqslant \lambda_1 \leqslant \dots
\leqslant \lambda_n$.

Similarly to Wigner's Theorem, the classical Marchenko-Pastur theorem states that, if $\frac{m}{n} \to \rho
\geqslant 1$ when $n$ goes to infinity, the empirical spectral measure
$L_{m,n}=\frac{1}{n}\sum_{j=1}^n\delta_{\lambda_j}$ converges almost surely to a deterministic measure,
called the Marchenko-Pastur distribution of parameter $\rho$. This measure is compactly supported and is
absolutely continuous with respect to Lebesgue measure, with density
\[ d\mu_{MP(\rho)}(x)=\frac{1}{2\pi
x}\sqrt{\big(b_{\rho}-x\big)\big(x-a_{\rho}\big)}\mathbbm{1}_{[a_{\rho},b_{\rho}]}(x)dx, \]
where $a_{\rho}=(1-\sqrt{\rho})^2$ and $b_{\rho}=(1+\sqrt{\rho})^2$ (see for example \cite{BaSi_2010_book}). Two different behaviors arise according to the value of $\rho$. If $\rho>1$, $a_{\rho}>0$ is a soft edge, which means that an eigenvalue $\lambda_j$ can be larger or smaller than $a_{\rho}$. On the contrary, if $\rho=1$, $a_{\rho}=0$ is called a hard edge: no eigenvalue can be less than $a_{\rho}$. Furthermore, the Marchenko-Pastur density function explodes at $0$. It is the case in particular when $m=n$.
We will denote by $\mu_{m,n}$ the approximate Marchenko-Pastur distribution whose density is defined by
\[\mu_{m,n}(x) =
\frac{1}{2\pi x}\sqrt{(x-a_{m,n})(b_{m,n}-x)}\mathbbm{1}_{[a_{m,n};\hspace{0.5mm} b_{m,n}]}(x), \]
with $a_{m,n}=(1-\sqrt{\frac{m}{n}})^2$ and $b_{m,n}=(1+\sqrt{\frac{m}{n}})^2$. 

The asymptotic behaviors of individual eigenvalues for LUE matrices have been known for some time and extended
to more general covariance matrices in the last decade. For an eigenvalue in the bulk of the spectrum, i.e.
$\lambda_j$ such that $\eta n \leqslant j \leqslant (1-\eta)n$, a Central Limit Theorem was proved by Su (see
\cite{Su_2006_fluctuations}) and Tao-Vu (see \cite{TaVu_2012_covariance}). From this theorem, the variance of
such eigenvalues is guessed to be of the order of $\frac{\log n}{n^2}$. For right-side intermediate
eigenvalues, i.e. $\lambda_j$ such that $\frac{j}{n} \to 1$ and $n-j \to \infty$, Su, and later Tao-Vu and
Wang, proved a CLT (see \cite{Su_2006_fluctuations}, \cite{TaVu_2012_covariance} and
\cite{Wa_2011_covariance_edge}). The variance appears to be of the order of $\frac{\log
(n-j)}{n^{4/3}(n-j)^{2/3}}$. Similar results probably hold for the left-side of the spectrum, when $\rho>1$.
Finally, for the smallest and the largest eigenvalues $\lambda_1$ and $\lambda_n$, CLTs were proved for
Gaussian matrices by Borodin-Forrester (see \cite{BoFo_2003_hard_soft_edge_transition}). Several authors
extended these results to more general covariance matrices. The latest results are due to Tao-Vu and Wang (see
\cite{TaVu_2012_covariance} and \cite{Wa_2011_covariance_edge}). Their variances are then guessed to be of the
order of $n^{-4/3}$. It should be mentionned that the result for the smallest eigenvalue only holds when
$\rho>1$. When $\rho=1$, in the case of a squared matrix, Edelman proved a CLT for the smallest eigenvalue
$\lambda_1$, extended by Tao and Vu in \cite{TaVu_2010_smallest_singular_value}, from which the variance is
guessed to be of the order of $n^{-4}$.

The following statement summarizes a number of quantitative bounds on the eigenvalues of covariance matrices
which are proved by methods similar to the ones developed in the preceding sections. For simplicity, we
basically assume that $\rho>1$. More precisely, we assume that $1<A_1 \leqslant \frac{m}{n} \leqslant A_2$
(where $A_1$ and $A_2$ are fixed constants) and that $S_{m,n}$ is a covariance matrix whose entries have an
exponential decay (condition $(C0)$) and have
the same first four moments as those of a LUE matrix.
\begin{thm}
\begin{enumerate}
 \item \textbf{In the bulk of the spectrum.}\\
 Let $\eta \in (0,\frac{1}{2}]$. There exist a constant $C>0$ (depending on $\eta$, $A_1$ and $A_2$) such that for all covariance matrix $S_{m,n}$, for all $\eta n\leqslant j \leqslant (1-\eta)n$,
 \[ \Var(\lambda_j) \leqslant C\frac{\log n}{n^2}.\]
 \item \textbf{Between the bulk and the edge of the spectrum.}\\
 There exists a constant $\kappa>0$ (depending on $A_1$ and $A_2$) such that the following holds. For all $K>\kappa$, for all $\eta \in (0,\frac{1}{2}]$, there exists a constant $C>0$~(depending on $K$, $\eta$, $A_1$ and $A_2$) such that for all covariance matrix $S_{m,n}$, for all $(1-\eta)n \leqslant j \leqslant n-K\log n$,
 \[ \Var(\lambda_j) \leqslant C\frac{\log (n-j)}{n^{4/3}(n-j)^{2/3}}.\]
 \item \textbf{At the edge of the spectrum.}\\
 There exists a constant $C>0$ (depending on $A_1$ and $A_2$) such that, for all covariance matrix $S_{m,n}$,
 \[ \Var(\lambda_n) \leqslant Cn^{-4/3}.\]
\end{enumerate}

\end{thm}
As for Wigner matrices, the first two results of this theorem are first proved in the Gaussian case, using the
fact that the eigenvalues of a LUE matrix form a determinantal process. It is then possible to derive a
deviation inequality for the eigenvalue counting function and then for individual eigenvalues. Integrating
leads to the results for LUE matrices. The result for the largest eigenvalue in the Gaussian case is already
known, see \cite{LeRi_2010_deviations}. These bounds are then extended to non-Gaussian matrices, relying on
three recent papers. First, Pillai and Yin proved in \cite{PiYi_2011_covariance} localization properties for
individual eigenvalues of covariance matrices, very similar to the localization properties for Wigner matrices
established by Erdös, Yau and Yin in \cite{ErYaYi_2010_rigidity}. Secondly, Tao and Vu (see
\cite{TaVu_2012_covariance}) and later Wang (see \cite{Wa_2011_covariance_edge}) proved a Four Moment Theorem
for these matrices. Combining these theorems as for Wigner matrices yield the theorem.

\subsection*{Acknowledgements}
I would like to thank my advisor, Michel Ledoux, for bringing up this problem to me, for the several
discussions we had about this work, and Emmanuel Boissard for very useful
conversations about Wasserstein distances.

\noindent Sandrine Dallaporta\\
Institut de Math\'ematiques de Toulouse, UMR $5219$ du CNRS\\
Universit\'e de Toulouse, F-$31062$, Toulouse, France\\
sandrine.dallaporta@math.univ-toulouse.fr

\end{document}